\documentclass[11pt,reqno]{amsart}
\usepackage{amssymb, bbm, color}
\usepackage[makeroom]{cancel} 
\usepackage{ulem}
\usepackage{amsfonts}
\usepackage{mathrsfs}
\usepackage{amsmath}
\usepackage{graphicx}
\usepackage{hyperref}
\usepackage{float}
\usepackage{epstopdf}
\usepackage{color}
\usepackage{bm}
\usepackage{comment}
\usepackage{soul}
\usepackage{esint}
\usepackage[shortlabels]{enumitem}
\usepackage{eufrak}

\allowdisplaybreaks
\newtheorem{theorem}{Theorem}[section]
\newtheorem{proposition}[theorem]{Proposition}
\newtheorem{lemma}[theorem]{Lemma}
\newtheorem{corollary}[theorem]{Corollary}
\newtheorem{remark}[theorem]{Remark}

\def\mcE{\mathcal{E}}

\def\mcQ{\mathcal{Q}}

\def\R{{\mathbbm R}}

 \makeatletter
\@namedef{subjclassname@2020}{%
  \textup{2020} Mathematics Subject Classification}
\makeatother

\def\eps{\varepsilon}
\def\wt{\widetilde}

\def\<{\langle}
\def\>{\rangle}

\numberwithin{equation}{section}

\begin{document}
\title[whether $p$-conductive homogeneity holds depends on $p$]{whether $p$-conductive homogeneity holds depends on $p$}

\author{Shiping Cao}
\address{Department of Mathematics, University of Washington, Seattle, WA 98195, USA}
\email{spcao@uw.edu}
\thanks{}

 \author{Zhen-Qing Chen}
\address{Department of Mathematics, University of Washington, Seattle, WA 98195, USA}\email{zqchen@uw.edu}
\thanks{}


\date{}


\begin{abstract}
We introduce two fractals, in Euclidean spaces of dimension two and three respectively, such the $2$-conductive homogeneity holds but there is some $\eps \in (0, 1)$ so that the $p$-conductive homogeneity fails for every $p\in (1, 1+\eps)$. In addition, these two fractals have Ahlfors regular conformal dimension within the interval $(1, 2)$ and $(2, 3)$,  respectively.
\end{abstract}

\maketitle

\section{Introduction}\label{sec1}
Some new progress \cite{CQ,CQ2,shi,Ki2} has been made in recent years on the construction of $p$-energies and therefore Sobolev space ${\mathcal W}^{1, p}$ on Sierpi\'{n}ski-like fractals for $p\in (1, \infty)$,   based on the framework of Kusuoka-Zhou \cite{KZ}. The idea is to define the $p$-energy forms as the $\Gamma$-limit of discrete $p$-energies on graph partitions of the fractals. To show the existence of a good limit, a crucial step is to verify certain inequalities of effective conductances, which are called  conditions (B1) and (B2) in    Kusuoka-Zhou \cite{KZ} when $p=2$. In a recent work, 
  Kigami   \cite[Definition 3.4]{Ki2}
 introduced a  $p$-conductive homogeneity  condition for $p>1$   
 as the $p$-counterpart   of  \cite[conditions (B1) and (B2)]{KZ}; see the paragraph following \cite[Definition 1.2 on p.6]{Ki2}. 
 This $p$-conductive homogeneity  condition plays an important role for some key properties of the Sobolev spaces  ${\mathcal W}^{1, p}$ 
 defined in \cite{Ki2}.

\medskip

It is a natural  question if $p$-conductive homogeneity of a compact metric space $K$ holds for some $p\in (1,\infty)$, then it holds for all  $p\in (1,\infty)$. It is shown recently by Murugan and  Shimizu \cite[Theorem C.28]{MS} that $p$-conductive homogeneity holds for the standard planar Sierpi\'{n}ski carpet equipped with the self-similar measure with the equal weight for any $p\in (1, \infty)$,   where the associated covering system is chosen to be the set of all pairs of cells of the same level that share a common border line. In this paper, we will show, however, that this is not true for general compact metric spaces. We show that there are  two fractals $F^{(2)}$ and  $F^{(3)}$ in dimension 2 and 3, respectively, so that `$p$-conductive homogeneity' holds for $p=2$ but fails for   $p\in (1, 1+\eps)$ for some $\eps \in (0, 1)$ in the sense of Remark \ref{R:1.2}.  
The fractal $F^{(3)}$ has Ahlfors regular conformal dimension   strictly larger than 2.
 To circumvent the issue about  the correct definition of neighbor disparity constants, we use  capacity 
(effective conductance) estimates 
to show that for each of these two fractals,
 there is some $\eps \in (0, 1)$ so that $p$-conductive homogeneity can not hold for   $p\in (1, 1+\eps)$ and for any of its covering systems.

\medskip 

We now describe these two fractals  $F^{(2)}$ and  $F^{(3)}$  in detail.  Fractal $F^{(2)}$ is an example of unconstrained planar Sierpi\'{n}ski  carpets considered in Cao and Qiu \cite{CQ},  while $F^{(3)}$ is an example of unconstrained Sierpi\'{n}ski carpets   in $\R^3$  studied in  Cao and Qiu  \cite{CQ2}.  

For $d\geq 1$,  let $F_0^{(d)}:=[0,1]^d$ be the unit cube in $\R^d$ and set $\mathcal{Q}^{(d)}_0:=\{F_0^{(d)} \}$. 
For each integer $n\geq 1$,   divide $F_0^{(d)}$ into $5^{nd}$ identical non-overlapping sub-cubes with side length $5^{-n}$.
 Denote the collection of such cubes by $\mcQ^{(d)}_n$:
\begin{equation}\label{e:1.1a}
	\mcQ^{(d)}_n:=\Big\{\prod_{i=1}^d[(l_i-1)/5^n,l_i/5^n]:1\leq l_i\leq 5^n,i=1,\cdots,d\Big\}.
\end{equation}
For each $A\subset \R^d$ and $n\geq 0$,   define 
\begin{equation}\label{e:1.2a}
	\mcQ_n^{(d)}(A):=\{Q\in\mcQ_n^{(d)}: \operatorname{int}(Q)\cap A\neq \emptyset\},
\end{equation}
where $\operatorname{int}(Q)$ stands for
the interior of the closed cube $Q$ in $\R^d$. 

Next, define $F^{(d)}_1$ by erasing from $F^{(d)}_0$ all cubes in $\mcQ^{(d)}_n$ that are attached to the center cube $[2/5, 3/5]^d$ with a $d-1$ dimensional face:  
\[
F_1^{(d)}:=F_0^{(d)}\setminus\Big(\bigcup_{i=1}^d({2}/{5},   {3}/{5})^{i-1}\times\big(( {1}/{5},  {2}/{5})\cup ( {3}/{5},  {4}/{5})\big)
\times( {2}/{5}, {3}/{5})^{d-i}\Big).
\] 
See Figure \ref{fig1} for the picture of $F^{(2)}_1$. Define
$F_n^{(d)}:=\bigcup_{Q\in \mcQ_1^{(d)}(F_1^{(d)})} \Psi_Q(F^{(d)}_{n-1})$ for $n\geq 2$, where, for each $Q\in \bigcup_{n=0}^\infty\mcQ^{(d)}_n$,  $\Psi_Q$ is the orientation preserving affine map from  $F^{(d)}_0$ onto $Q$.  The fractals that we are interested in are 
\[
F^{(d)}:=\bigcap_{n=0}^\infty F^{(d)}_n
\]
with $d=2,3$. See Figure \ref{fig1} for a picture of an approximation of $F^{(2)}$. Note that $F^{(d)}$ is not a generalized Sierpi\'{n}ski carpet  in the sense of \cite[\S 2.2]{BBKT} as the interior connectedness condition (H2) there
is not satisfied.

\begin{figure}[htp]
	\includegraphics[width=4cm]{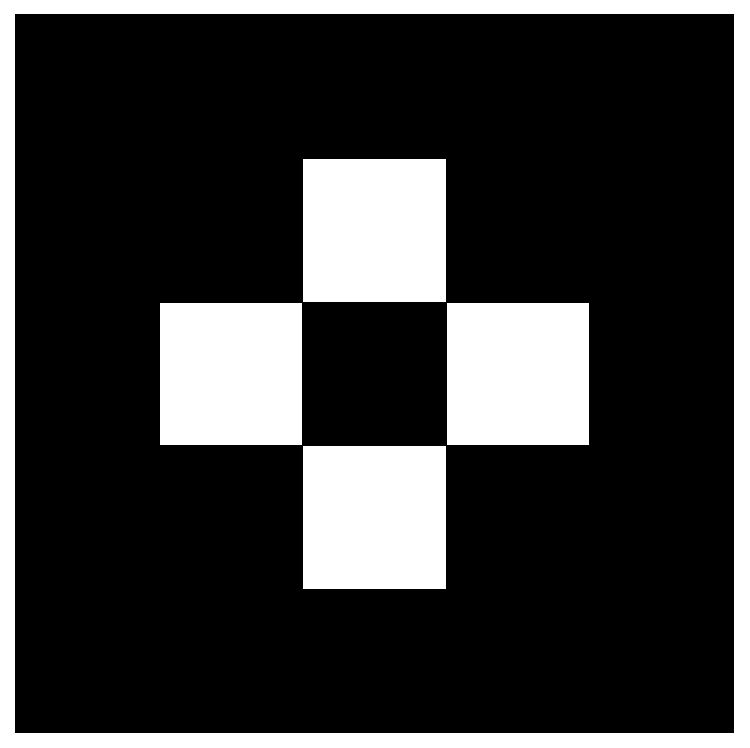}\qquad
	\includegraphics[width=4cm]{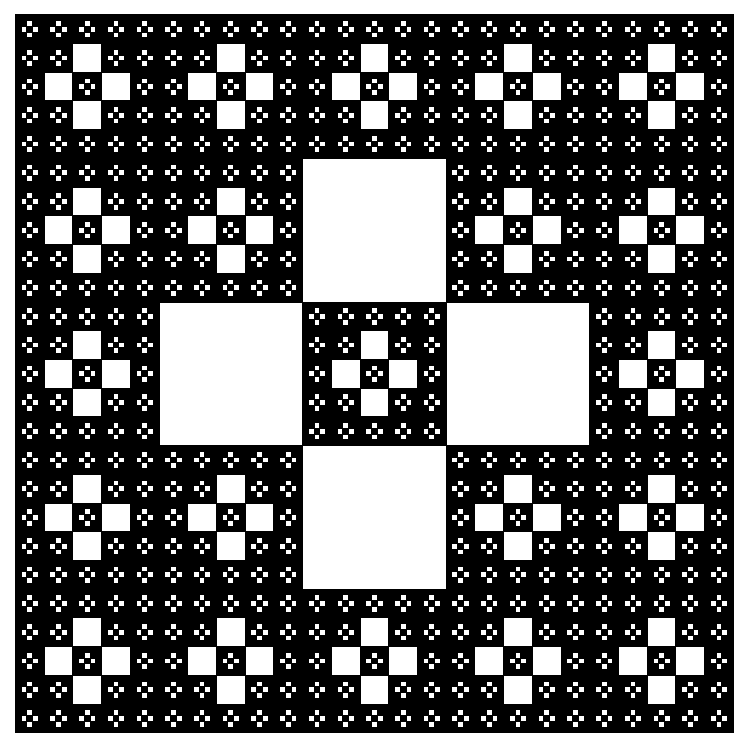}
	\caption{$F_1^{(2)}$ and $F^{(2)}$}\label{fig1}
\end{figure}

Under the Euclidean metric, $F^{(d)}$ is a compact $\alpha$-Ahlfors regular set with respect to the Hausdorff measure on $F^{(d)}$, where 
$\alpha$ is the Hausdorff dimension of $F^{(d)}$, that is, 
\begin{equation}\label{e:1.3a}
\alpha = {\rm dim}_H F^{(d)} = \frac{\log (5^d-2d)}{\log 5}.
\end{equation}

We now  introduce a natural partition of the metric measure space $(F^{(d)},\mu)$, where $\mu$ is the normalized Hausdorff measure on $F^{(d)}$ so that $\mu (F^{(d)})=1$.

\medskip 

\noindent \textbf{Partition of $F^{(d)}$}.  There is a natural partition in the sense of  Kigami \cite[Definition 2.3]{Ki2}, explained as follows. Let $T=\bigcup_{n=0}^\infty \mathcal{Q}_n^{(d)}(F^{(d)})$, and let $\mathcal{A}$ be the subset of $T\times T$ such that $(Q,Q')\in \mathcal{A}$ if and only if $Q\subset Q'$ with $Q\in \mathcal{Q}_{n+1}^{(d)}(F^{(d)}),Q'\in \mathcal{Q}_n^{(d)}(F^{(d)})$ for some $n\geq 0$, or $Q'\subset Q$ with $Q\in \mathcal{Q}_n^{(d)}(F^{(d)}),Q'\in \mathcal{Q}_{n+1}^{(d)}(F^{(d)})$ for some $n\geq 0$. Then, $(T,\mathcal{A},F^{(d)}_0)$ is a rooted tree, where $T$ is the set of vertices, $\mathcal{A}$ is the set of edges and $F^{(d)}_0$ is the root. We assign each $Q\in T$ the subset $\Psi_Q(F^{(d)})$ of $F^{(d)}$. One can check that $\{\Psi_Q(F^{(d)}); Q\in T\}$ is a partition of $F^{(d)}$ that satisfies \cite[Assumption 2.15]{Ki2} with $M_*=M_0=1$. 
  
  \medskip

We next  define the discrete $p$-energy forms for $p\in (1, \infty)$ and effective $p$-conductances.

\medskip

\noindent \textbf{$p$-energy forms}. For $d\geq 2$ and $n\geq 1$,  define the discrete $p$-energy forms on $l\big(\mcQ^{(d)}_n(F^{(d)})\big)$ by
\[
\mathcal{E}^{n}_p(f)=\frac{1}{2}\sum_{Q,Q'\in \mcQ^{(d)}_n(F^{(d)})\atop Q\cap Q'\neq\emptyset}\big(f(Q)-f(Q')\big)^p\quad\hbox{ for each }f\in l\big(\mcQ^{(d)}_n(F^{(d)})\big).
\]

\medskip

\noindent \textbf{Effective $p$-conductances}. For each $n,m\geq 0$ and $A\subset \mcQ^{(d)}_n(F^{(d)})$, define 
\begin{equation}\label{e:1.3}
S^m(A):=\{Q\in  \mcQ^{(d)}_{n+m}(F^{(d)}):Q\subset   Q'\hbox{ for some }Q'\in A  \}.
\end{equation}
For $n\geq 1$ and disjoint $A_1,A_2\subset \mcQ^{(d)}_n(F^{(d)})$, define
\[
\mathcal{E}_{p,m}(A_1,A_2):=\inf \left\{\mathcal{E}_{p}^{n+m}(f):\,f\in l\big(\mcQ^{(d)}_{n+m}(F^{(d)})\big),\,f|_{S^m(A_1)}=1,\,f|_{S^m(A_2)}=0
\right\}.
\]
For short, if $A_1=\{Q\}$ for some $Q\in \mcQ^{(d)}_n(F^{(d)})$ and $A_2\subset \mcQ^{(d)}_n(F^{(d)})$, we write $\mathcal{E}_{p,m}(Q,A_2)$ 
for $\mathcal{E}_{p,m}(\{Q\},A_2)$.

In the notation of \cite{Ki2}, we have $\mathcal{E}_{1,p,m}\big(Q,\mathcal{Q}^{(d)}_n(F^{(d)})\big)=\mathcal{E}_{p,m}\big(Q,\Gamma(Q)^c\big)$ for each $n\geq 1$ and $Q\in \mathcal{Q}^{(d)}_n(F^{(d)})$. We note that \cite[Assumption 2.15(5)]{Ki2} is just \cite[Assumption 2.7]{Ki2}, and  (1)-(4) of \cite[Assumption 2.15]{Ki2} imply 
Assumptions 2.6, 2.10 and 2.12 of \cite{Ki2} by \cite[Proposition 2.16]{Ki2}. Since   $F^{(d)}$  satisfies \cite[Assumption 2.15]{Ki2}
with  partition $\{\Psi_Q(F^{(d)}); Q\in T\}$, we have the following from \cite[Theorem 3.30]{Ki2}.

\begin{lemma} \label{lemma1}  
 For $p>1$, if $F^{(d)}$  is $p$-conductive homogeneous with respect to some covering system in the sense of \cite[Definition 3.4]{Ki2}, then the following holds.  
	
\noindent{\rm ($\textbf{A}_p$)}. There exist some positive constants $\sigma >0$ 
  and $ c_1, c_2 >0$ so that for each $n\geq 1,m\geq 0$ and $Q\in \mcQ^{(d)}_n(F^{(d)})$, 
	\[
	c_1\sigma^{-m}\leq \mathcal{E}_{p,m}(Q,\Gamma(Q)^c)\leq c_2\sigma^{-m},
	\]
	where $\Gamma(Q):=\{Q'\in \mcQ^{(d)}_n(F^{(d)}):\ Q'\cap Q\neq\emptyset\}$ and $\Gamma(Q)^c:=\mcQ^{(d)}_n(F^{(d)})\setminus \Gamma(Q)$.
\end{lemma}
 
 \medskip

\begin{remark}\label{R:1.2}  \rm 
In \cite[Definition 3.4]{Ki2}, the definition of $p$-conductive homogeneity of a compact metric space $(K, \rho)$ involves the class of neighbor disparity constants
that  depends on the covering system $\mathscr{J}$ used; cf.  \cite[Definition 2.29 and p.35]{Ki2}.

In this paper, we say that $p$-conductive homogeneity   fails for a compact metric space $(K, \rho)$ if for any  covering system $\mathscr{J}$ (in the sense \cite[Definition 2.29]{Ki2}),  the corresponding $p$-conductive homogeneity condition for $(K, \rho)$ fails. Otherwise, we say that $p$-conductive homogeneity condition holds for $(K, \rho)$.   \qed 
\end{remark}

The following theorem is the main result of this paper,  whose proof will be given in  next section.
 
\begin{theorem}\label{thm1}
On $F^{(d)}$, property  {\rm ($\textbf{A}_p$)} fails  for $p\in (1, \frac{\log 10}{\log 5})$ when $d=2$, and fails  for $p\in (1,\frac{\log16}{\log5})$ when $d=3$. 
\end{theorem}

 \begin{corollary}\label{C:1.4} 
For $F^{(2)}$, the $p$-conductive homogeneity condition    holds for   $p>\operatorname{dim}_{AR}(F^{(2)},\rho)$  with the covering system 
\[
\mathscr{J}^{(2)}=\left\{\{Q,Q'\}:\,\{Q,Q'\}\subset\mcQ^{(2)}_n(F^{(2)})\hbox{ for some }n\geq 1,\,Q\neq Q',\, Q\cap Q'\neq\emptyset\right\},
\]
however $p$-conductive homogeneity condition fails for $p\in (1, \frac{\log 10}{\log 5})$.

For $F^{(3)}$, $p$-conductive homogeneity condition holds for $p=2$ with the covering system 
\[
\mathscr{J}^{(3)}=\left\{\{Q,Q'\}:\,\{Q,Q'\}\subset\mcQ^{(3)}_n(F^{(3)})\hbox{ for some }n\geq 1,\,Q\neq Q',\,\#(Q\cap Q')>1\right\},
\]
however $p$-conductive homogeneity condition fails for $p\in (1,\frac{\log16}{\log5})$.

Moreover, 
\begin{equation}\label{e:1.4} 
 \frac{\ln 10}{\ln 5} \leq \operatorname{dim}_{AR}(F^{(2)},\rho)  \leq \frac{\log 21}{\log 5} \quad \hbox{and} \quad 
 \frac{\ln 80}{\ln 5} \leq \operatorname{dim}_{AR}(F^{(3)},\rho)  \leq \frac{\log 119}{\log 5},
\end{equation}
where  $\operatorname{dim}_{AR} \big(F^{(d)},\rho \big)$ stands for  the Ahlfors regular conformal dimension  of the metric space $ F^{(d)}$ equipped with Euclidean metric $\rho$; see Remark \ref{R:2.2} for its definition. 
\end{corollary}

\begin{proof}
For $F^{(2)}$, the first claim is due to \cite[Condition (B) and its proof on page 18]{CQ}   for $p=2$ and the same proof of \cite{CQ} also works for $p>\operatorname{dim}_{AR}(F^{(2)},\rho)$,  while the second claim is a consequence of Lemma \ref{lemma1} and Theorem \ref{thm1}.

For $F^{(3)}$, the first claim is due to \cite[Theorem 8.1]{CQ2}, while the second claim is a consequence of Lemma \ref{lemma1} and Theorem \ref{thm1}

 Assertion \eqref{e:1.4} will be proved in Remark \ref{R:2.2}.
\end{proof} 

\begin{remark}\rm  
When $d\geq 3$, we think  the proof of \cite{CQ} can be suitably modified to show that the $p$-conductive homogeneity holds for every $p> \operatorname{dim}_{AR}(F^{(d)},\rho)$ with the covering system 
\[
\mathscr{J}^{(3)}=\left\{\{Q,Q'\}:\,\{Q,Q'\}\subset\mcQ^{(3)}_n(F^{(3)})\hbox{ for some }n\geq 1,\,Q\neq Q',\,\#(Q\cap Q')>1\right\}.
\]
However, we do not pursue this extension in this paper.  \qed 
\end{remark}

\begin{remark} \rm
The fractal $F^{(3)}$ in particular gives an example of a compact metric space for which the $p$-conductive homogeneity condition holds for some $p =2<
\operatorname{dim}_{AR}(F^{(3)},\rho)$ and fails for any other  
$p \in (1, \frac{\log16}{\log5})$, which is also smaller than $\operatorname{dim}_{AR}(F^{(3)},\rho)$. 
 \end{remark}

 We can say more about about the $p$-conductive homogeneity condition for $F^{(2)}$ by comparing it with a closely related generalized
   Sierpi\'{n}ski carpet  $\wt F^{(2)}$ in $\R^d$ to be defined below. 
   Let $\wt F^{(2)}_0= [0, 1]^2$,  $\wt F_1^{(2)}=[0,1]^2\setminus\big((\frac{2}{5},\frac{3}{5})\times (\frac{1}{5},\frac{4}{5})\bigcup (\frac{2}{5},\frac{3}{5})\times (\frac{1}{5},\frac{4}{5})\big)\subset F_1^{(2)}$, and let $\wt F_n^{(2)}=\bigcup_{Q\in \mcQ^{(2)}_1(\wt F_1^{(2)})}\Psi_Q(\wt F_{n-1}^{(2)})$ for $n\geq 2$. Here as in the above,  $ \Psi_Q$ is the orientation preserving affine map from  $\wt F^{(2)}_0$ onto $Q$. Then $\wt F^{(2)}:=\bigcap_{n=1}^\infty F_n^{(2)}$  is a generalized Sierpinski carpet 
 in the sense of \cite[\S 2.2]{BBKT} having Hausdorff dimension ${\rm dim}_H (\wt F^{(2)}) = \frac{\log 20}{\log 5}$. 
See Figure \ref{fig4} for $\wt F^{(2)}$ and $\wt F_1^{(2)}$.  \begin{figure}[htp]
\includegraphics[width=4cm]{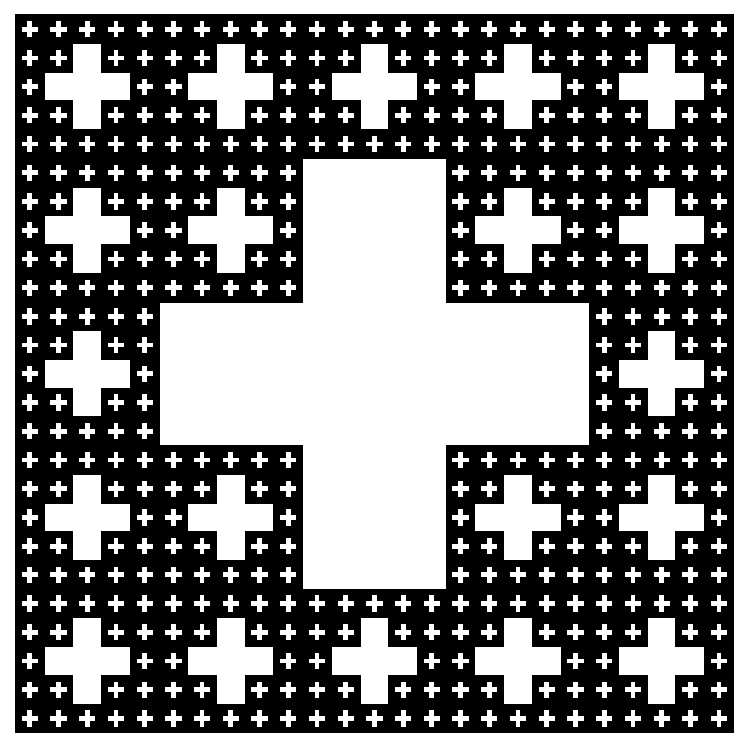}\quad 
\includegraphics[width=4cm]{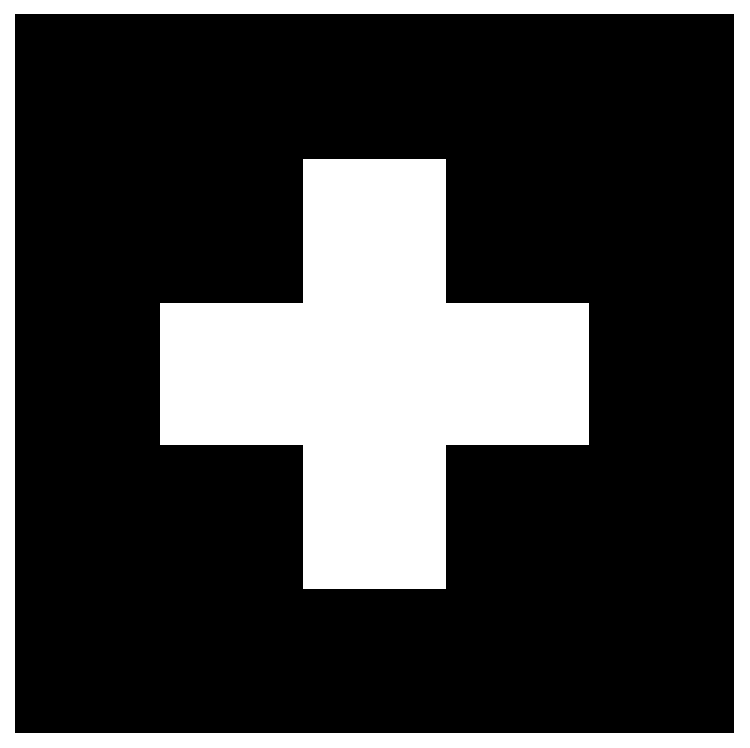}
\caption{The generalized carpet $F$ and its first level approximating $F_1$.}\label{fig4}
\end{figure}

\begin{proposition}\label{P:1.7} 
 Property  {\rm ($\textbf{A}_p$)} fails for $F^{(2)}$ for  $1<p<\dim_{AR}(\wt F^{(2)},\rho)$. 
Consequently, the  $p$-conductive homogeneity   fails for $F^{(2)}$ for  $1<p<\dim_{AR}(\wt F^{(2)},\rho)$. 
\end{proposition}

\medskip

The proof of Proposition \ref{P:1.7} will be given in Section \ref{S:2}. We conclude this section with two open questions. 
\begin{enumerate}
\item[(i)] It can be shown that $\dim_{AR}(  F^{(2)},\rho) \geq \dim_{AR}(\wt F^{(2)},\rho)$ but we do not know they are the same or not.
We suspect they are. If they are, then Proposition \ref{P:1.7} combined with Corollary \ref{C:1.4} would imply that 
the $p$-conductive homogeneity   holds on  $F^{(2)}$ for  $p> \dim_{AR}( F^{(2)},\rho)$ but fails for $1<p< \dim_{AR}( F^{(2)},\rho)$. 
 
\item[(ii)] Corollary \ref{C:1.4} and Proposition \ref{P:1.7} raise a natural question: if the $p$-conductive homogeneity   holds
on a compact metric space $(K, \rho)$ for some $p>1$, does the $q$-conductive homogeneity hold on $(K, \rho)$ for any $q>p$?
This looks quite plausible but we do not have a solution for it.  The second part of Corollary \ref{C:1.4} shows that on a 
compact metric space $(K, \rho)$ that the $p$-conductive homogeneity  fails for some $p>1$,
 the smallest $q$ that the $q$-conductive homogeneity holds on $(K, \rho)$ is in general different from the Ahlfors regular conformal dimension
 $\dim_{AR}( K,\rho)$ of $K$. 
\end{enumerate}

\section{$p$-conductive homogeneity} \label{S:2} 

In this section, we present the proof for Theorem \ref{thm1} and Proposition \ref{P:1.7}. In the following two lemmas, we consider two cells $Q_1$ and $Q_2$, and deduce some estimates of the effective $p$-conductances. For $Q\in \mcQ^{(d)}_n(F^{(d)})$, recall the definition of $\Gamma (Q)$ from Lemma \ref{lemma1}. For each $n,m\geq 0$ and $A\subset \mcQ^{(d)}_n(F^{(d)})$, recall the definition of $S^m(A)$ from \eqref{e:1.3}.

\begin{lemma}\label{lemma21}
Let $Q_1=[0, {1}/{5}]^d$ (see Figure \ref{fig2} for an illustration). Then 
\[
\begin{aligned}
\mathcal{E}_{p,m}(Q_1,\Gamma(Q_1)^c)&\geq 2^m\,(5^m+1)^{1-p}\quad &\hbox{ for }d=2 \hbox{ and } m\geq 1,\\
\mathcal{E}_{p,m}(Q_1,\Gamma(Q_1)^c)&\geq 16^m\,(5^m+1)^{1-p}\quad &\hbox{ for } d=3 \hbox{ and }  m\geq1.
\end{aligned}
\]
\end{lemma}
\begin{figure}[htp]
	\includegraphics[width=4cm]{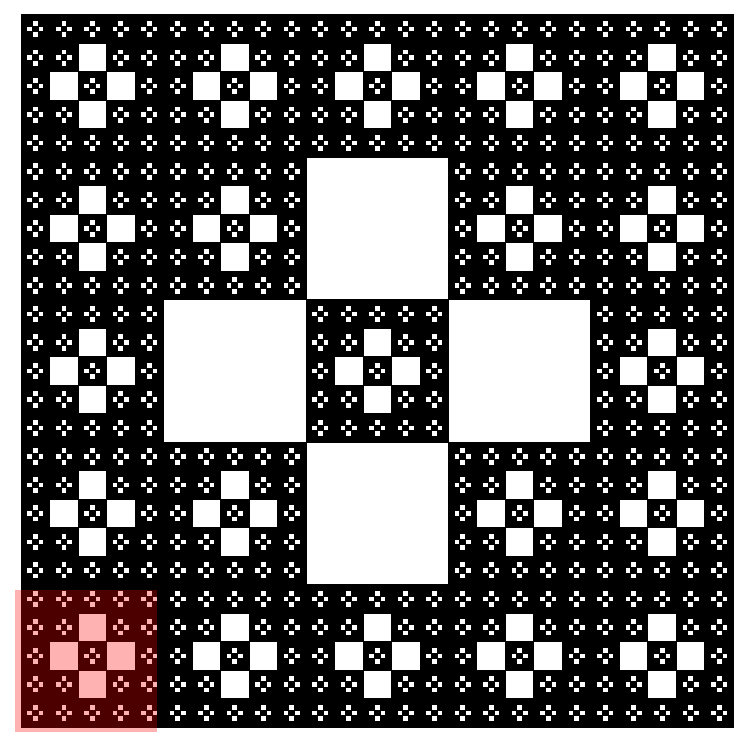}
	\caption{$Q_1$ marked red}\label{fig2}
\end{figure}

\begin{proof}
For $d\geq 1$, we define $G^{(d)}_1$ as 
\[
G^{(d)}_1:=\bigcup \left\{Q\in \mathcal{Q}^{(d)}_1(F^{(d)}_1):\,Q\cap \partial F_0^{(d)}\neq \emptyset \right\},
\]
and inductively,   define for $n\geq 2$,  
\[
G^{(d)}_n=\bigcup_{Q\in\mcQ^{(d)}_1(G^{(d)}_1)}\Psi_Q(G_{n-1}^{(d)})\ \hbox{ and }\ G^{(d)}=\bigcap_{n=0}^\infty G^{(d)}_n.
\]
Note that when $d=1$, $G^{(d)}$ is a Cantor set; when $d=2$, $G^{(d)}$ is a generalized Sierpi\'{n}ski carpet in the sense of \cite[\S 2.2]{BBKT}. Moreover, 
\[
[0,1]\times G^{(d-1)}\subset F^{(d)}\quad\hbox{ for } d\geq 2.
\]

Let $h\in l\big(\mcQ^{(d)}_{m+1}(F^{(d)})\big)$ be a function that satisfies
\[
h|_{S^m(Q_1)}=1 \quad \hbox{ and } \quad h|_{S^m(\Gamma(Q_1)^c)}=0.
\]
For each $\wt{Q}\in \mcQ^{(d-1)}_{m+1}(G^{(d-1)}\cap [0,1/5]^{d-1})$, consider the path of cells 
\[
Q_{\wt{Q},i}=\Big[\frac{5^m-1+i}{5^{m+1}},\frac{5^m+i}{5^{m+1}} \Big]\times\wt{Q}\quad  \hbox{for  }  i=0,1,\cdots,5^m+1.
\]
In particular, one can check that $Q_{\wt{Q},0}\in S^m(Q_1)$, $Q_{\wt{Q},5^m+1}\in S^m\big(\Gamma(Q_1)^c\big)$ and $Q_{\wt{Q},i}\in \mcQ^{(d)}_{m+1}([0,1]\times G^{(d-1)})\subset \mcQ^{(d)}_{m+1}(F^{(d)})$ for each $i=0,1,\cdots,5^{m+1}$.
Then, we have 
\[
\begin{aligned}
\mathcal{E}_p^{m+1}(h)&\geq \sum_{\wt{Q}\in \mcQ^{(d-1)}_{m+1}(G^{(d-1)}\cap [0,1/5]^{d-1})}\sum_{i=0}^{5^m}\big(h(Q_{\wt{Q},i+1})-h(Q_{\wt{Q},i})\big)^p\\
&\geq \#\mcQ^{(d-1)}_{m+1}(G^{(d-1)}\cap [0, {1}/{5}]^{d-1})\cdot (5^m+1)^{1-p},
\end{aligned}
\]
where in the second inequality, we used the fact that $h(Q_{\wt{Q},0})=1$ and $h(Q_{\wt{Q},5^{m}+1})=0$ for each $\wt{Q}\in \mcQ^{(d-1)}_{m+1}(G^{(d-1)}\cap [0,1/5]^{d-1})$ and the H\"older's inequality that 
$$
 (M+1)^{-1} \big|\sum_{k=0}^M a_k \big|   \leq \Big( \sum_{k=0}^M |a_k|^p /(M+1) \Big)^{1/p}. 
$$
The lemma follows immediately after noticing that $\#\mcQ^{(d-1)}_{m+1}(G^{(d-1)}\cap [0,1/5]^{d-1})=2^m$ when $d=2$, and  $\#\mcQ^{(d-1)}_{m+1}(G^{(d-1)}\cap [0,1/5]^{d-1})=16^m$ when $d=3$. 
\end{proof}

\begin{remark} \label{R:2.2} \rm 
		 Lemma \ref{lemma21} has the following  geometric implication.
		  By \cite[Theorems 4.7.6 and 4.9.1]{Ki1}, for $F^{(d)}$ with $d=2,3$ and the Euclidean metric $\rho$, we know that 
		\[
		\limsup_{m\to\infty}\max_{Q\in \bigcup_{n\geq 1}\mcQ^{(d)}_n(F^{(d)})}\big(\mathcal{E}_{p,m}(Q,\Gamma(Q)^c)\big)^{1/m}<1\quad\hbox{ if and only if }\quad p>\dim_{AR}(F^{(d)},\rho),
		\]
		where $\operatorname{dim}_{AR}(F^{(d)},\rho)$ is the Ahlfors regular conformal dimension of the metric space $(F^{(d)},\rho)$, that is, 
		\begin{eqnarray*}
			\operatorname{dim}_{AR}(F^{(d)},\rho)&:=&\inf\{\alpha: \mbox{there exists a metric $\rho'$ on $F^{(d)}$ that is quasi-symmetric to $\rho$ and}\\ 
			&& \hskip 0.1truein \mbox{a Borel regular measure $\mu'$ that is $\alpha$-Ahlfors regular with respect to $\rho'$.}\}
		\end{eqnarray*}
      This together with Lemma \ref{lemma21} implies that 
		\[\operatorname{dim}_{AR}(F^{(2)},\rho)  \geq   1+ \frac{\ln 2}{\ln 5} = \frac{\ln 10}{\ln 5}
		\quad\hbox{ and }  \quad\operatorname{dim}_{AR}(F^{(3)},\rho) \geq 1+ \frac{\ln 16}{\ln 5} = \frac{\ln 80}{\ln 5}.
		\]
On the hand, since $(F^{(d)}, \rho)$ is ${\rm dim}_H F^{(d)}$-Ahlfors regular with the  Hausdorff measure on $F^{(d)}$,
we have by the definition of Ahlfors regular conformal dimension and \eqref{e:1.3a} that 
$$
\operatorname{dim}_{AR}(F^{(d)},\rho) \leq {\rm dim}_H (F^{(d)},\rho) = \frac{\log (5^d-2d)}{\log 5}.
$$ 
 Consequently, we get 
 $$
  \frac{\ln 10}{\ln 5} \leq \operatorname{dim}_{AR}(F^{(2)},\rho)  \leq \frac{\log 21}{\log 5} \quad \hbox{and} \quad 
 \frac{\ln 80}{\ln 5} \leq \operatorname{dim}_{AR}(F^{(3)},\rho)  \leq \frac{\log 119}{\log 5}.
$$
\qed
\end{remark}

\begin{lemma}\label{lemma22}
	Let $Q_2=[{2}/{5},  {3}/{5}]^d$ (see Figure \ref{fig3} for an illustration). We have 
	\[
	\begin{aligned}
		\mathcal{E}_{p,m}(Q_2,\Gamma(Q_2)^c)&\leq 4  \qquad &\hbox{for }d=2,   \hbox{ and }  m\geq1, \\
		\mathcal{E}_{p,m}(Q_2,\Gamma(Q_2)^c)&\leq 7(12\cdot 5^m-16)  \quad  &\hbox{for }d=3  \hbox{ and }  m\geq1. 
	\end{aligned}
	\]
\end{lemma}
\begin{figure}[htp]
	\includegraphics[width=4cm]{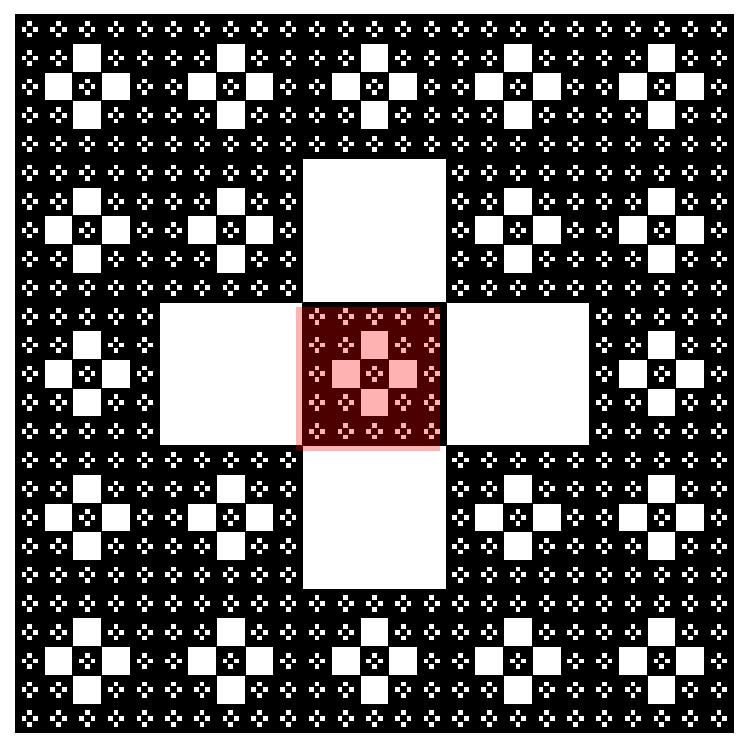}
	\caption{$Q_2$ marked red}\label{fig3}
\end{figure}
\begin{proof}
Define $f\in l(\mcQ^{(d)}_{m+1}(F^{(d)}))$ by
\[
f(Q)=\begin{cases}
	1\quad\hbox{ if }Q\in \mcQ^{(d)}_{m+1}(F^{(d)}\cap Q_2),\\
	0\quad\hbox{ if }Q\in \mcQ^{(d)}_{m+1}(F^{(d)}\setminus Q_2).
\end{cases}
\] 
Then 
\[
\mathcal{E}_{p,m}(Q_2,\Gamma(Q_2)^c)\leq \mathcal{E}_p^{m+1}(f)=\sum_{Q\in S^m(Q_2)}\sum_{Q'\in S^m(\Gamma(Q_2)^c)\atop Q\cap Q'\neq\emptyset}\big(f(Q)-f(Q')\big)^p.
\]
For $d=2$, there are only $4$ $Q\in S^m(Q_2)$ such that $Q'\cap Q\neq\emptyset$ for some $Q'\in S^m(\Gamma(Q_2)^c)$, and each $Q$ intersects exactly with one such $Q'$, so we have $\mathcal{E}_{p,m}(Q,\Gamma(Q)^c)\leq \mathcal{E}_p^{m+1}(f)=4$; for $d=3$, there are $12\cdot 5^m-16$ $Q\in S^m(Q_2)$ such that  $Q'\cap Q\neq\emptyset$ for some $Q'\in S^m(\Gamma(Q_2)^c)$ (that is all the cells attached to the boundary edges of the $Q$), and each $Q$ intersects with at most $7=2^3-1$ $Q'$, so we have $\mathcal{E}_{p,m}(Q_2,\Gamma(Q_2)^c)\leq  \mathcal{E}_p^{m+1}(f)\leq 7(12\cdot 5^m-16)$.
\end{proof}

\medskip

\begin{proof}[Proof of Theorem \ref{thm1}] For $d=2$, when $1< p<1+\frac{\log 2}{\log 5}$, we have  
$$
 \lim_{m\to \infty} \frac{\mathcal{E}_{p,m}(Q_1,\Gamma(Q_1)^c)}{\mathcal{E}_{p,m}(Q_2,\Gamma(Q_2)^c)}= \infty,
 $$
  where $Q_1,Q_2  \in \mcQ^{(2)}_{1}(F^{(2)})$ are the cells in the statements of Lemmas \ref{lemma21} and \ref{lemma22},
  respectively. Hence ($\textbf{A}_p$) can not
   hold for $p\in (1,\frac{\log 10}{\log 5})$. 

For $d=3$, when $1< p<\frac{\log 16}{\log 5}$, we have  that $ \lim_{m\to \infty}  \frac{\mathcal{E}_{p,m}(Q_1,\Gamma(Q_1)^c)}{\mathcal{E}_{p,m}(Q_2,\Gamma(Q_2)^c)} =  \infty$, where $Q_1,Q_2 \in\mcQ_{1}^{(3)}(F^{(3)})$ are the cells of Lemmas \ref{lemma21} and \ref{lemma22}. Hence ($\textbf{A}_p$) can not hold for $p\in (1, \frac{\log 16}{\log 5})$.
\end{proof}

\bigskip

\begin{proof}[Proof of Proposition \ref{P:1.7}]  
For each $p>1$,  define the discrete $p$-energy $\wt \mcE^{m}_{p}$ on $\mcQ^{(2)}_m(\wt F^{(2)})$ by 
\[
\wt \mcE^{m}_{p}(f)=\sum_{Q,Q'\in \mcQ^{(2)}_m(\wt F^{(2)})\atop 
Q\cap Q'\neq\emptyset}\big(f(Q)-f(Q')\big)^p.
\] 
Let $\nu$ be  the normalized Hausdorff measure on $\wt F^{(2)}$ such that $\nu(\wt F^{(2)})=1$
For $n\geq 0$ and $f\in L^p(\wt F^{(2)};\nu)$,  define 
$ P_nf\in l\big(\mcQ^{(2)}_n(\wt F^{(2)})\big)$ by 
$$
P_nf(Q)= \frac1{\nu \big( Q\cap \wt F^{(2)} \big)} \int_{\wt F^{(2)}\cap Q}f(w)\nu(dw) 
\quad  \hbox{for each }  Q\in \mcQ^{(2)}_n(\wt F^{(2)}). 
$$ 
For each $n\geq 1$, $m\geq 0$ and $Q\in \mcQ^{(2)}_n(\wt F^{(2)})$, we define  
\[
\begin{aligned}
\wt \mcE_{p,m}\big(Q,\Gamma(Q)^c\big):=\inf \Big\{\wt \mcE^{m+n}_p(P_{m+n}f):\,f&\in L^p(\wt F^{(2)};\nu),\,f|_{Q}=1,\,f|_{Q'}=0\\
&\hbox{for each }Q'\in \mcQ^{(2)}_n(\wt F^{(2)})\hbox{ such that }Q'\cap Q=\emptyset
\Big\}.
\end{aligned}
\] 
According to  \cite[Theorem 10.2 and Remark 10.20]{MS}, there is some $\wt \sigma_p>0$
 so that 
\begin{equation}\label{e:2.1}
\wt \mcE_{p,m}\big(Q,\Gamma(Q)^c\big)\asymp \wt \sigma_p^{-m}
\hbox{ for every }m\geq 1,\,Q\in \mcQ_n^{(2)} (\wt F^{(2)}).
\end{equation}
Moreover, for $Q_1=[0, {1}/{5}]\times [0, {1}/{5}]$  and $p>1$, by \cite[Theorem 6.17, Theorem 10.2 and Remark 10.20]{MS}, we can find $f_p \in C(\wt F^{(2)})$ such that 
$  f_p|_{Q_1}=1$, $\ f_p|_{Q'}=0$  for each  $Q'\in\mcQ^{(2)}_1(\wt F^{(2)})$ with $Q'\cap Q_1=\emptyset$ and that 
$$
\wt \mcE^{m}_p(P_mf_p)\asymp \wt \sigma_p^{-m}  \quad \hbox{for all } m\geq 1.
$$
As a consequence, we have for each $p>p'>1$, 
$$
\frac{\wt \sigma_{p}^{-m}}{\wt \sigma_{p'}^{-m}}\lesssim \frac{\wt \mcE^{m}_p(P_mf_{p'})}{\wt \mcE^{m}_{p'}(P_mf_{p'})}\leq \sup_{Q,Q'\in \mcQ_m^{(2)}(\wt F^{(2)})\atop Q\cap Q'\neq\emptyset}|P_mf_{p'}(Q)-P_mf_{p'}(Q')|^{p-p'}\to 0
$$
as $m\to\infty$. This implies that $\wt \sigma_p$ is strictly increasing in $p\in (1, \infty)$.  Moreover, we can easily check that $\wt F^{(2)}$ satisfies \cite[Assumption 2.15]{Ki2} with $M_*=1$. Hence by by \eqref{e:2.1} and \cite[Proposition 3.3]{Ki2}, we know that $\wt \sigma_p>1$ if and only if $p>\dim_{AR}(\wt F^{(2)},\rho)$, which together with the fact that $\wt \sigma_p$ is strictly increasing implies  $\wt \sigma_p<1$ if $p<\dim_{AR}(\wt F^{(2)},\rho)$.  Noticing that $\mcQ_{m+1}^{(2)}(\wt F^{(2)})\subset \mcQ_{m+1}^{(2)}(F^{(2)})$ for each $m\geq 0$, we see that for 
$1<p<\operatorname{dim}_{AR}(\wt F^{(2)},\rho)$, 
\begin{equation}\label{eqn21}
\lim\limits_{m\to\infty}\mcE_{p,m}\big(Q_1,\Gamma(Q_1)\big)\geq\lim\limits_{m\to\infty}\wt \mcE_{p,m}\big(Q_1,\Gamma(Q_1)\big)
\gtrsim \lim\limits_{m\to\infty} \wt \sigma_p^{-m} =\infty. 
\end{equation}
 This together with Lemma \ref{lemma22} yields that property {\rm ($\textbf{A}_p$)} fails for $F^{(2)}$ when $1<p<\operatorname{dim}_{AR}(\wt F^{(2)},\rho)$. 
  \end{proof}

 \hskip 0.2truein
 
\end{document}